\documentclass{amsart}
\usepackage{mathtools}
\usepackage{mathabx}
\usepackage{amsmath, amscd, amsfonts}
\usepackage{xcolor}
\usepackage{tikz} %This package offers the ability to draw pictures
\usepackage[mathscr]{eucal}
\usepackage{indentfirst}
\usepackage{pict2e}
\usepackage{epic}
\numberwithin{equation}{section}
\usepackage{epstopdf} 

\usepackage{accents}
\usepackage{tikz-cd}

\usepackage{bbm}
\usepackage{lipsum}
\usepackage{graphicx}
\usepackage{dashrule}
\usepackage{hyperref}
\hypersetup{
	colorlinks,
	citecolor=blue,
	filecolor=blue,
	linkcolor=blue,
	urlcolor=blue
}
\usepackage{tcolorbox}
\usepackage{todonotes}
\usepackage{esint}
\usepackage{datetime}
\usepackage{wrapfig}

\usepackage{enumitem}

\usepackage[backend=biber]{biblatex}
\usepackage{adjustbox}
\theoremstyle{plain}
\newtheorem{theorem}{Theorem}
\newtheorem*{theorem*}{Theorem}
\newtheorem{lemma}[theorem]{Lemma}
\newtheorem*{lemma*}{Lemma}
\newtheorem{corollary}[theorem]{Corollary}
\newtheorem*{corollary*}{Corollary}

\newtheorem*{proposition*}{Proposition}

\theoremstyle{definition}
\newtheorem{definition}[theorem]{Definition}
\newtheorem*{definition*}{Definition}

\newtheorem*{conjecture*}{Conjecture}

\newtheorem*{convention*}{Convention}
\newtheorem{remark}[theorem]{Remark}
\newtheorem*{remark*}{Remark}

\newtheorem*{problem*}{Problem}
\newtheorem{example}[theorem]{Example}
\newtheorem*{example*}{Example}

\newtheorem*{exercise*}{Exercise}

\newtheorem*{fact*}{Fact}

\newtheorem*{notation*}{Notation}

\newtheorem*{observation*}{Notation}

\newcommand{\C}{\mathbb{C}}

\newcommand{\CCC}{\mathcal{C}}

\newcommand{\DDD}{\mathcal{D}}

\newcommand{\FFF}{\mathcal{F}}
\newcommand{\GGG}{\mathcal{G}}
\newcommand{\HHH}{\mathcal{H}}

\newcommand{\LLL}{\mathcal{L}}

\newcommand{\NNN}{\mathcal{N}}

\newcommand{\WWW}{\mathcal{W}}
\newcommand{\Z}{\mathbb{Z}}
\newcommand{\N}{\mathbb{N}}

\newcommand{\colim}{\text{colim}}

\newcommand{\knotlinewidth}{1pt}
\tikzset{overcross/.style={double, line width=1.5, white, double=black, double distance=\knotlinewidth}}

\newcommand{\Hom}{{\rm{Hom}}}

\definecolor{mygreen}{RGB}{0,128,0}
\definecolor{bubbles}{rgb}{0.91, 1.0, 1.0}

\makeatletter
\renewcommand{\thesection}{\@arabic\c@section}
\makeatother
\makeatletter
\newcommand*{\relrelbarsep}{.386ex}
\newcommand*{\relrelbar}{%
\mathrel{%
	\mathpalette\@relrelbar\relrelbarsep
}%
}
\newcommand*{\@relrelbar}[2]{%
\raise#2\hbox to 0pt{$\m@th#1\relbar$\hss}%
\lower#2\hbox{$\m@th#1\relbar$}%
}
\providecommand*{\rightrightarrowsfill@}{%
\arrowfill@\relrelbar\relrelbar\rightrightarrows
}
\providecommand*{\leftleftarrowsfill@}{%
\arrowfill@\leftleftarrows\relrelbar\relrelbar
}
\providecommand*{\xrightrightarrows}[2][]{%
\ext@arrow 0359\rightrightarrowsfill@{#1}{#2}%
}
\providecommand*{\xleftleftarrows}[2][]{%
\ext@arrow 3095\leftleftarrowsfill@{#1}{#2}%
}
\makeatother

\begin{document}

\begin{abstract}
	\noindent We prove an essentially surjective Galois-correspondence-like functor for $n$-stacks. More specifically, it gives an essentially surjective functor from the $\infty$-category of $n$-stacks of finite sets with an action of the fundamental group of $X$ to the $\infty$-category of Deligne-Mumford $n$-stacks \'etale over a connected scheme $X$.
\end{abstract}

\title{The ``Galois Correspondence'' for $n$-Stacks}
\author{Yuxiang Yao}
\email{yuxiany@uci.edu} 
\address{Department of Mathematics\\University of California-Irvine}
% 'also optional'
%\classification{18N55 , 14A30, 18E35, 55P60 (primary). }

\maketitle

%\author{F. Author}
%\email{f.author@some.where.ac.uk}  % 'optional'
%\address{Mathematics Department\\Some University\\Where Road\\%
	%	Here\\MT55 9XX}
%\curraddr{Mathematics Institute\\Another University\\
	%	There\\BA1 1HZ}          % 'also optional'

%\tableofcontents

\section{Introduction}\label{Section of Introduction}
In this note, we will introduce an essentially surjective Galois-correspondence-like functor for $n$-stacks. Our motivation is to have an explicit write-up for such a functor for Deligne-Mumford stacks that are finite \'etale over a nice scheme $X$. But the literature actually deduces a more general version for Deligne-Mumford $n$-stacks rather than only for usual stacks. So we instead build the note for a more general situation. Let's first state our main theorem for groupoids below:

\begin{theorem}[Galois Correspondence for $n$-groupoids]\label{main theorem}
	Let $\eta$ be a geometric point in a connected scheme $X$. There is an exact equivalence of categories of fibrant objects
	\[F_\eta: n\text{-}\mathrm{Grpds}(\mathrm{FEt}_X)\xrightarrow{\sim} \pi_1^{\mathrm{\acute{e}t}}(X, \eta)\text{-}n\text{-}\mathrm{Grpds}(\mathrm{FinSets}).\]
	Here $n\text{-}\mathrm{Grpds}(\mathrm{FEt}_X)$ is the $1$-category of $n$-groupoids in $\mathrm{FEt}_X$, the category of finite \'etale (surjective) covers over $X$; and $\pi_1^{\mathrm{\acute{e}t}}(X, \eta)\text{-}n\text{-}\mathrm{Grpds}(\mathrm{FinSets})$ is the category of $n$-groupoids in finite sets with the \'etale fundamental group, $\pi_1^{\mathrm{\acute{e}t}}(X, \eta)$, acting on them. 
\end{theorem}

This will be proved in Section \ref{Section for main thm}. Then a weaker version of the above result for $n$-stacks can be considered as a corollary of the theorem above by simplicial localization. This will be concluded in Section \ref{Section for Corollary} after the proof of the main theorem and a bit more prerequisites. Note that $n$-groupoids can be considered as a nice device for presenting stacks, see \cite{P2013}. Also, note that we cannot say the following functor is an equivalence of simplicial categories. The subtleties will be explained in the Proof of Corollary \ref{Corollary_Main Corollary Proof} below as well as in the Example and Remarks before it in Section \ref{Section for Corollary}.

\begin{corollary}\label{corollary of main thm}
	Let $\eta$ be a geometric point in a connected scheme $X$. There exists an essentially surjective functor
	\[G^\LLL_\eta: \pi^{\mathrm{\acute{e}t}}_1(X, \eta)\text{-}n\text{-}\mathrm{Stacks(FinSets)}\rightarrow\mathrm{DM}\text{-}n\text{-}\mathrm{Stacks}/\mathrm{\acute{e}tale}\text{ }\mathrm{over}\text{ }X 
	\]
	Here, $\mathrm{DM}\text{-}n\text{-}\mathrm{Stacks}/\mathrm{\acute{e}tale}\text{ }\mathrm{over}\text{ }X$ is given by the simplicial localization of \'etale $n$-groupoids in $\mathrm{Et}(X)$; and $\pi^{\mathrm{\acute{e}t}}_1(X, \eta)\text{-}n\text{-}\mathrm{Stacks(FinSets)}$ is obtained by the simplicial localization of $\pi^{\mathrm{\acute{e}t}}_1(X, \eta)\text{-}n\text{-}\mathrm{Grpds}(\mathrm{FinSets})$.
\end{corollary} 

Note that the category DM-$n$-Stacks/finite \'etale over $X$ of DM-$n$-stacks finite \'etale over $X$ is NOT the full subcategory of $n$-Stacks(Et$_X$) consists of all DM-$n$-Stacks finite \'etale over $X$, since it has fewer hypercovers.

We may call it \textit{Galois functor for $n$-stack}, since it is not actually a weak equivalence of $\infty$-categories. The story here can be traced back to a well-known result in the context of schemes, suppose that $X$ is a connected scheme with a geometric point $\eta$, one has a schematic version of Galois correspondence:

\begin{theorem}\label{SchemeGaloisCorrespondence}
	\textit{The following functor, called the fiber functor,}
	\[F^{\mathrm{Sch}}_\eta: \mathrm{FEt}(X)\to \pi^{\mathrm{\acute{e}t}}_1(X, \eta)\mathrm{-FinSets},\;\; (Y\xrightarrow{f} X) \mapsto \eta^* Y, \]
	\textit{is an equivalence of categories.}
\end{theorem}

For more details, see Expos\'e V in \cite{SGA1} or \href{https://stacks.math.columbia.edu/tag/0BND}{Tag 0BND} in \cite{stacks-project}.

This gives a fundamental relation between the finite covers of a given scheme $X$ and the finite sets with actions of the \'etale fundamental group of $(X, \eta)$. For families of connected degree $d$ branched covers of projective line $\mathbb{P}^1_R$ with $n$ branched points over a connected normal scheme $X$, Theorem \ref{SchemeGaloisCorrespondence} plays an important role in constructing a moduli space. Here, we consider (connected normal) schemes $X$ of finite type over $R = \C$ or $\Z_p$ and with fixed monodromy group $G$. This moduli space is called a \textit{Hurwitz scheme}, denoted by $\mathrm{Hur}_{G, n, d}^{\mathbf{c}}$. To avoid ambiguity, let's explicitly define an isomorphism above from a connected cover $p: Y\to \mathbb{P}^1$ to another connected cover $q: Y'\to \mathbb{P}^1$ by homeomorphism $\phi: Y\to Y'$ such that $q\circ \phi = p$. For automorphisms of a connected cover $p: Y\to \mathbb{P}^1$, they are special cases when $p = q$ above and form a group $\mathrm{Aut}(p)$. Note that in the above setting, we are considering mere covers with monodromy group $G$ rather than Galois $G$-covers. In the literature, many authors use automorphisms with extra conditions, which is often called the \textit{inner automorphisms}; our weaker automorphism here are often called \textit{absolute automorphism}. See, for example, \cite{FV91}, for details. 

However, in the construction of $\mathrm{Hur}_{G, n, d}^{\mathbf{c}}$, a moduli {\em scheme} can be constructed only when the centralizer of the monodromy group $G$ in the symmetric group $S_d$ is trivial. To solve the problem, one can consider coarse moduli, or construct the moduli in the ``category'' of Deligne-Mumford stacks. We would like to use the latter strategy, which is our motivation to build Theorem \ref{main theorem} and Corollary \ref{corollary of main thm}. This moduli stack will be called a Hurwitz stack, and only requires the case of $1$-stacks in Corollary \ref{corollary of main thm}. Here we also note that a key reference cite by many papers on Hurwitz stacks is \cite{RW06}. 

We will introduce the necessary preliminaries so that the above Theorem \ref{main theorem} makes sense at the end of this note. For the classical materials above, one may refer to, for example, Theorem 5.4.2 in \cite{S09}. An analogous topological version of such functor in Corollary \ref{corollary of main thm} will be briefly introduced at the end of the notes to close this story.

Finally, we give a brief summary of the literature. On algebraic geometry side, we will instead consider the \'etale groupoid schemes as atlases for, equivalently as presentations of, Deligne-Mumford stacks in this note. This maybe less used than other interpretations of Deligne-Mumford stacks. We will give a brief explanation in Remark \ref{remark equivalence of different interpretations of DM-stacks}, we also refer to Section 4.3 in \cite{LMB18} in order to connect this article to another interpretation of DM-stacks. In this note, in the proof of Corollary \ref{corollary of main thm} from Theorem \ref{main theorem} largely depends on passing to simplicial localization. The primary reference for Hammock localization is the paper by Dwyer and Kan \cite{DK80b}, which is easier to use than another earlier model of simplicial localization in \cite{DK80a} also by Dwyer and Kan. For an accessible introduction to the material in homotopy theory, a good reference is ``The Homotopy Theory of $(\infty, 1)$-Categorie'' by Bergner\cite{Bergner2018}, and the simplicial localization is introduced in section 3.5.

\subsection*{Acknowledgement} I am grateful to my advisor, Jesse Wolfson, for introducing me to the subject, extensive comments on each of my multiple drafts, as well as for his guidance and supports on all possible perspectives. I would like to thank Craig Westerland for helpful suggestions and comments. I would like to thank Aaron Landesman, Ronno Das, Oliver Braunling and Oishee Banerjee for reading one of my earlier drafts and helpful comments. The authors thank Chenchang Zhu for her thoughts on potential generalizations. I would like to thank Julie Bergner for her suggestions and detailed comments in my earlier draft. Finally, I am indebted to anonymous reviewers of an earlier edition of this paper, who pointed out a subtle mistake that I failed to catch by myself.

\section{The Category of $n$-Groupoids in a Category $\CCC$}

For this section, we take Behrend and Getzler \cite{BG17} and Wolfson \cite{W16} as our primary references. Let's first determine the category $\CCC$ we will work on:

\begin{definition}[Kernel pairs and Effective Epimorphisms]$ $
	\begin{enumerate}
		\item The \textbf{kernel pair} of a morphism $f: X\to Y$ in a category with finite limits is the pair of parallel arrows 
		\[X\times_Y X\rightrightarrows X.\]
		\item The map $f: X\to Y$ is an \textbf{effective epimorphism} if $f$ is the coequalizer of this pair
		\[X\times_Y X\rightrightarrows X\xrightarrow{f} Y.\]
	\end{enumerate}
\end{definition}

From now, we will work in the category $\CCC$ with the following axioms:
\begin{definition}[Category with covers  and finite limits]
	A \textbf{category $\CCC$ with covers and finite limits} consists of a category $\CCC$ with a subcategory of covers satisfying the following five axioms:
	\begin{enumerate}[label=$(\mathrm{C\arabic*})$, start=0]
		\item $\CCC$ has finite limits.
		\item The category $\CCC$ has a terminal object $e$, and the map $X\to e$ is a cover for every object $X\in \CCC$. 
		\item Pullbacks of covers along arbitrary maps exist and are covers. 
		\item If $g$ and $f$ are composable, and $f$ and $gf$ are both covers, then $g$ is also a cover. 
		\item Covers are effective epimorphisms. 
	\end{enumerate}
\end{definition}
\begin{remark}
	This is a union of the axioms in ``the category with covers'' in Wolfson \cite{W16} and the ``descent category'' in Behrend and Getzler \cite{BG17}. Then we are able to make use of results in both papers for the above category $\CCC$. Those categories we care about in this note will also fortunately fit into the two categories in both papers.
\end{remark} 

We define that $[k] = \{0<1<\cdots<n\}$ is a nonempty linearly ordered set, and $\Delta$ is a category whose objects are such $[k]$, and morphisms are nondecreasing maps; define $\Delta^k:= \mathrm{Mor}_{\Delta}(-, [k])$. A simplicial object in $\CCC$ is a functor $X_\bullet: \Delta^{\mathrm{op}}\to \CCC$, and a simplicial set $S$ is a simplicial object in $\CCC = \mathrm{Sets}$. Equivalently, a simplicial set $S$ can be written as the colimit of its simplices
\[S\cong \colim_{\Delta^k\to S} \Delta^k.\]

We will also need the following simplicial subsets of $\Delta^k$:
\begin{itemize}
	\item The boundary $\partial \Delta^k\subset\Delta^k$ can be defined by $(\partial \Delta^k)([m]) = \{\alpha\in \mathrm{Mor}_\Delta ([m], [k])\::\: \alpha\text{ is not surjective}\}$.
	\item The $i$-th horn $\Lambda_i^k\subset \partial \Delta^k$ can be defined as $(\Lambda_i^k)([m]) = \{\alpha\in \mathrm{Mor}_\Delta ([m], [k])\::\: [k]\not\subseteq \alpha([m])\cup \{i\}\}$.
\end{itemize}

Let $X$ be a simplicial object in $\CCC$, in other words, a functor $X: \Delta^{\mathrm{op}}\to \CCC$; and let $S$ be a simplical set. Dual to the colimit expression of a simplicial set, we define 
\[\textrm{Hom}(S, X):= \lim_{\Delta^k \to S} X_k.\]

\begin{definition}[$n$-groupoids, Definition 2.5, Wolfson]
	Let $n\in \N\cup \{\infty\}$. A \textbf{$n$-groupoid} is a simplicial object $X\in \mathrm{s}\CCC$ such that for all $k>0$ and $0\leq i\leq k$, $\Lambda_i^k(X)$ exists in $\CCC$ and the induced map
	\[\lambda_i^k (X): X_k\to \Hom(\Lambda_i^k, X)\]
	is a cover, and it is also an isomorphism for $k> n$. \\
	Here the map $\lambda_i^k (X)$ is induced from the simplicial subset $\Lambda_i^k\subset \Delta^k$. 
\end{definition}

\begin{definition}[$n$-Hypercovers, Definition 2.6, Wolfson]\label{Definition_Hypercovers}$ $
	\begin{enumerate}
		\item Let $n\in \N\cup\{\infty\}$. A map $f: X\to Y$ of simplicial objects in $\mathrm{s}\CCC$ is an \textbf{$n$-hypercover} if, for all $k\geq 0$, the map 
		\[\mu_k(f): X_k\to \Hom(\partial \Delta^k, X)\times_{\Hom(\partial \Delta^k , Y)} Y_k\]
		is a cover for all $k$, and it is an isomorphism for $k\geq n$. \\ 
		Here, the pullback on the right is given by two induced maps  to $\Hom(\partial \Delta^k , Y)$ from (1) $f:X\to Y$ and (2) $\partial \Delta^k\subset \Delta^k$; and the induced map $\lambda_k(f)$ is determined by the map $f_k: X_k\to Y_k$ and another map $X\to  \Hom(\partial \Delta^k, X)$ induced again by $\partial \Delta^k\subset \Delta^k$.
		\item We refer to an $\infty$-hypercover simply as a \textbf{hypercover}.
	\end{enumerate}
\end{definition}

% Now, we may define $n$-Grpds$(\CCC)$ be the $1$-category of $n$-groupoids in $\CCC$, a category with covers and finite limit. An equivalence $F: \CCC\to \DDD$ of categories with covers and finite limits is a category equivalence preserving covers. We left the following lemma as an exercise:

% \begin{lemma}\label{Equivalence of Cats implies Equivalence of n Grpds}
	% Suppose that $F: \CCC\to \DDD$ is an equivalence of categories with covers and finite limits. It induces an equivalence of $1$-categories:
	% \[F: n\text{-}\mathrm{Grpds}(\CCC)\xrightarrow{\sim}n\text{-}\mathrm{Grpds}(\DDD).\]
	% % In addition, if $\CCC$ and $\DDD$ are categories with covers and $F: \CCC\to \DDD$ preserves covers, then $F$ is exact.
	% \end{lemma}

\section{Categories of Fibrant Objects}

In order to get Theorem \ref{main theorem} as well as its corollar, Corollary \ref{corollary of main thm}, we need more prerequisites. First, we introduce Brown's category of fibrant objects as below (See the Definition below in Section 1 of \cite{BG17}). 

We take Behrend and Getzler\cite{BG17} and Rogers and Zhu \cite{RZ20} as our primary references.

\begin{definition}[Category of Fibrant Objects]
	A \textbf{category of fibrant objects} is a small category $\CCC$ with the subcategory of weak equivalences $\WWW\subseteq \CCC$ and a subcategory $\FFF\subseteq \CCC$ of fibrations, with the following axioms:
	\begin{enumerate}[label=$(\mathrm{F\arabic*})$, start=1]
		\item There exists a terminal object $e\in \mathrm{Obj}(\CCC)$, and a morphism $X\to e$ in $\FFF$ for each object $X\in \mathrm{Obj}(\CCC)$;
		\item Pullbacks of fibrations are fibrations;
		\item Pullbacks of trivial fibrations are trivial fibrations, where trivial fibrations are morphisms in both $\WWW$ and $\FFF$.
		\item Every morphism $f: X\to Y$ has a factorization
		\begin{small}
			\begin{center}
				\begin{tikzcd}
					X \arrow[dr, "r"']\arrow[rr, "f"]&& Y\\
					&P\arrow[ur, "q"']&
				\end{tikzcd}
			\end{center}
		\end{small}
		where $r$ is in $\WWW$ and $q$ is in $\FFF$.
	\end{enumerate}
	Objects in this category are said to be \textbf{fibrant}. And we will use \textbf{CFO} for ``category of fibrant objects''.
\end{definition}
Note that we only need the existence of the terminal object $e\in \mathrm{Obj}(\CCC)$, which does not further require a concrete choice. 

\begin{definition}[Exact Functors between CFOs]
	A functor $F: \CCC\to \DDD$ between CFOs is called an \textbf{exact functor} if 
	\begin{enumerate}
		\item[(1)] $F$ preserves terminal objects, fibrations and trivial fibrations; 
		\item[(2)] Any pullback square of the form
		\begin{center}
			\begin{tikzcd}
				X\times_Z Y\arrow[d, "g"']\arrow[r]&X\arrow[d, "f"]\\
				Y\arrow[r, "h"']&Z
			\end{tikzcd}
		\end{center}
		with $f$ being a fibration in $\CCC$ is mapped to a pullback square in $\DDD$, i.e., 
		$F(X\times_Z Y)\cong F(X) \times_{F(Z)} F(Y)$.
	\end{enumerate}
\end{definition}
The following lemma is from Roger and Zhu \cite{RZ20}, Lemma 2.9.
\begin{lemma}
	An exact functor $F: \CCC\to \DDD$ between CFOs preserves finite products and weak equivalences.
\end{lemma}

\begin{definition}
	An \textbf{equivalence $F: \CCC\to \DDD$ of categories with covers and finite limits} is a category equivalence preserving covers. 
\end{definition}

As an important example in the note, we may define $n$-Grpds$(\CCC)$ be the $1$-category of $n$-groupoids in $\CCC$, a category with covers and finite limit. We leave the following lemma as an exercise:
\begin{lemma}\label{Equivalence of Cats implies Equivalence of n Grpds}
	Suppose that $F: \CCC\to \DDD$ is an equivalence of categories with covers and finite limits. It induces an exact equivalence of ($1$-)categories:
	\[F: n\text{-}\mathrm{Grpds}(\CCC)\xrightarrow{\sim}n\text{-}\mathrm{Grpds}(\DDD).\]
\end{lemma}

%\label{Section for main thm}

We introduce a few more definitions and results for $n$-groupoids and also prove our main theorem and its corollary in this section. 

\begin{definition}[Categories with Weak Equivalences]
	A category $\CCC$ is called a \textbf{category with weak equivalence} if there is a distinguished set $\WWW$ of morphisms in $\CCC$, whose elements are called \textbf{weak equivalences}, such that
	\begin{itemize}
		\item $\WWW$ contains all the isomorphisms in $\CCC$;
		\item $\WWW$ has $2$ out of $3$ property, in other words, given $f: X\to Y$, $g: Y\to Z$ and the composition $gh: X\to Z$ in $\CCC$, if two out of three above are in $\WWW$, the third is also in $\WWW$.
	\end{itemize}
	Denote the category with weak equivalences by a pair $(\CCC, \WWW)$ (or just $\CCC$ if the context is clear).
\end{definition}

\begin{definition}[Weak Equivalences for $n$-Groupoids]$ $
	\begin{enumerate}
		\item Let $P_n: \mathrm{s}\CCC\to \mathrm{s}\CCC$ be a functor of simplicial objects such that 
		$(P_nY)_m:= \Hom(\Delta^{m, n}, Y)$,
		where $\Delta^{m, n} = \Delta^m\times \Delta^n$.
		\item A morphism $f:X\to Y$ of $k$-groupoids is a \textbf{weak equivalence} if the fibration
		$q(f): P(f)\to Y$
		is a hypercover, where $P(f) = X\times_{f, Y, d_1^*} P_1Y$, and $q(f)$ is the pullback of $f$ along $d_0^*: P_1Y\to Y$. %\footnote{Here, $(\partial_0)_m, (\partial_1)_m: (P_1Y)_m\to Y_m$ are originated from the two natural maps $\Delta^m\to \Delta^{m, 1}$, which are induced by the two functions $[0]\to [1]$ mapping $0$ to $0$, $1$, respectively.}
	\end{enumerate}
\end{definition}

\begin{lemma}[Proposition 6.7, Rogers and Zhu]\label{Lem3-25onBG}
	A fibration $f: X\to Y$ of $n$-groupoids is a weak equivalence if and only if it is a hypercover.
\end{lemma}

The following is a special case of Theorem 3.6 on Behrend-Getzler \cite{BG17}:

\begin{lemma}\label{n groupoids form a CFO}
	Suppose that $\CCC$ is a category with covers and finite limits. With weak equivalences as above definition, the category of $n$-groupoids in $\CCC$, i.e., $n$-$\mathrm{Grpds}(\CCC)$ is a CFO. Here, the fibrations in $n$-$\mathrm{Grpds}(\CCC)$ are given by morphisms $f: X\to Y$ of $n$-groupoid objects in $\CCC$ with the condition: for all $k\geq 0$, the map 
	\[X_k\to \Hom(\Lambda_i^k, X)\times_{\Hom(\Lambda_i^k , Y)} Y_k\]
	is a cover for all $k$, and for all $0\leq i\leq k$.  
\end{lemma}

\begin{remark}
	Together with the previous lemma, we obtain that trivial fibrations are the same as hypercovers.
\end{remark}

Recall a simplicial group $G$ is a simplicial object in the category of groups. We have a natural way to define the simplicial group action $G$ on a simplicial set $X$. One can refer to Section V.2, Goerss and Jardine \cite{GJ09} for more details. Then, one can immediately define categories like $G$-sSets, $G$-$n$-Grpds, $G$-$n$-FinGrpds and $G$-$n$-stacks.

The following lemma is an exercise in the above definitions:
\begin{lemma}\label{Pull Out Simplicial group action out of n-Gprds(-)}
	Suppose that $\CCC$ is a category with covers and finite limits. Then we have an equivalence of usual categories:
	\[G\text{-}n\text{-}\mathrm{Gprds}(\mathrm{FinSets})\simeq n\text{-}\mathrm{Gprds}(G\text{-}\mathrm{FinSets}).\]
\end{lemma}

\section{The Main Theorem and Its Corollary}\label{Section for Corollary}
\label{Section for main thm}
We now prove our main theorem.

\begin{theorem}[Galois Correspondence for $n$-groupoids]\label{Galois Correspodence n Groupoids}
	There is an equivalence of categories
	\[F_\eta: n\text{-}\mathrm{Grpds}(\mathrm{FEt}_X)\xrightarrow{\sim} \pi_1^{\mathrm{\acute{e}t}}(X, \eta)\text{-}n\text{-}\mathrm{Grpds}(\mathrm{FinSets}).\]
\end{theorem}

\begin{proof}
	By Theorem \ref{SchemeGaloisCorrespondence}, $$F^{\mathrm{Sch}}_\eta: \mathrm{FEt}_X\to \pi_1^{\mathrm{\acute{e}t}}(X, \eta)\text{-}\mathrm{FinSets}$$ is a category equivalence. By Lemma \ref{Equivalence of Cats implies Equivalence of n Grpds}, the functor $$n\text{-}\mathrm{Grpds}(F^{\mathrm{Sch}}_\eta): n\text{-}\mathrm{Grpds}(\mathrm{FEt}_X)\to n\text{-}\mathrm{Grpds}(\pi_1^{\mathrm{\acute{e}t}}(X, \eta)\text{-}\mathrm{FinSets})$$
	is an equivalence. By Lemma \ref{Pull Out Simplicial group action out of n-Gprds(-)}, there is an equivalence 
	$$G: n\text{-}\mathrm{Grpds}(\pi_1^{\mathrm{\acute{e}t}}(X, \eta)\text{-}\mathrm{FinSets})\xrightarrow{\sim} \pi_1^{\mathrm{\acute{e}t}}(X, \eta)\text{-}n\text{-}\mathrm{Grpds}(\mathrm{FinSets}).$$
	Define $F_\eta = G\circ \left(n\text{-}\mathrm{Grpds}(F^{\mathrm{Sch}}_\eta)\right)$, this gives the equivalence we want. 
	\\
	Finally, since $F^{\mathrm{Sch}}_\eta$ in Theorem \ref{SchemeGaloisCorrespondence} preserves covers, by Lemma \ref{Equivalence of Cats implies Equivalence of n Grpds}, $F_\eta$ is an equivalence. 
\end{proof}

%%%%%%%%%%%%%%%%%%
%%%%%%%%%%%%%%%%%%
%%%%%%%%%%%%%%%%%%
%%%%%%%%%%%%%%%%%%
%%%%%%%%%%%%%%%%%%
%%%%%%%%%%%%%%%%%%
%%%%%%%%%%%%%%%%%%
To show Corollary \ref{corollary of main thm}, we need more prerequisites about simplicial localization as below. 

Let $\CCC$ be a CFO with $\HHH$ be the subcategory of trivial fibrations, and $\WWW$ be the subcategory of weak equivalences. There are several different models of simplicial localizations. Note that Dwyer and Kan have two papers \cite{DK80a} and \cite{DK80b} on two different models of simplicial localizations. For a comprehensive introduction, one can read the book \cite{Bergner2018}, where the simplicial localizations are introduced in Section 3.5. In the two models, Hammock localization introduced by \cite{DK80b}, denoted by $\LLL^H(\CCC, \WWW)$ or simply by $\LLL(\CCC, \WWW)$ below, is easier to use. We will introduce hammock localization and another convenient model of simplicial localization denoted by $\HHH^{-1}\CCC$ below. Note that \cite{RZ20} introduces an equivalence between two models under good settings. 

\begin{enumerate}
	\item A widely used model is called
	``hammock localization'' of $\CCC$ with respect to $\WWW$, denoted by $\LLL^H(\CCC, \WWW)$, or just $\LLL(\CCC, \WWW)$ in this note. 
	Define $\LLL(\CCC, \WWW)$ to be a simplicial category whose objects are Obj$(\CCC)$, and for each pair of objects $X, Y$ in $\CCC$, $\LLL(\CCC, \WWW)(X, Y)$ is a simplicial set such that for each $k$, elements in $\LLL(\CCC, \WWW)(X, Y)_k$ are of the following format:
	
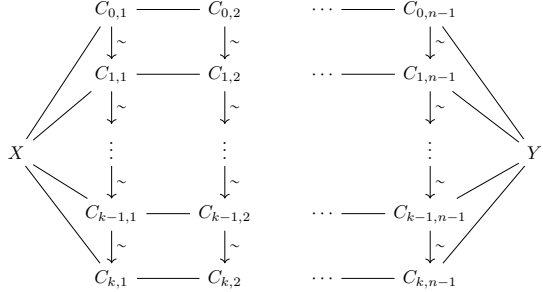
\begin{center}
		
		\begin{tikzcd}
			&C_{0, 1}\arrow[d, , "\sim"]\arrow[ddl, ,dash]\arrow[r, dash]&C_{0, 2}\arrow[d, "\sim"]&\cdots&C_{0, n-1}\arrow[ddr, dash]\arrow[l, dash]\arrow[d, "\sim"]&\\
			&C_{1, 1}\arrow[d, "\sim"]\arrow[dl, dash]\arrow[r, dash]&C_{1, 2}\arrow[d, "\sim"]&\cdots&C_{1, n-1}\arrow[dr, dash]\arrow[l, dash]\arrow[d, "\sim"]&\\
			X&\vdots\arrow[d, "\sim"]&\vdots\arrow[d, "\sim"]&&\vdots\arrow[d, "\sim"]&Y\\
			&C_{k-1, 1}\arrow[d, "\sim"]\arrow[ul, dash]\arrow[r, dash]&C_{k-1, 2}\arrow[d, "\sim"]&\cdots&C_{k-1, n-1}\arrow[ur, dash]\arrow[l, dash]\arrow[d, "\sim"]&\\
			&C_{k, 1}\arrow[uul, dash]\arrow[r, dash]&C_{k, 2}&\cdots&C_{k, n-1}\arrow[uur, dash]\arrow[l, dash]&
		\end{tikzcd}

	\end{center}
	where $\sim$ means morphisms in $\WWW$, and $n$ is the length, and $k$ is the height. 
	\begin{enumerate}
		\item Here, each row
		\begin{center}
			
			\begin{tikzcd}
				X \arrow[r, dash] & C_{i, 1} \arrow[r, dash, ] &\cdots \arrow[r, dash]& C_{i, n-1}\arrow[r, dash, ]& Y
			\end{tikzcd}
			
		\end{center}
		is a \textit{zigzag} from $X$ to $Y$ of length $n\geq 0$. Here, a \textit{zigzag }from $X$ to $Y$ is given by 
		\begin{center}
			\begin{tikzcd}
				X \arrow[r, equal]& C_0\arrow[r, dash, "f_1"] & C_1 \arrow[r, dash, "f_2"] &\cdots \arrow[r, dash, "f_n"]& C_n \arrow[r, equal]& Y
			\end{tikzcd}
		\end{center}
		where the line segments $f_i$ can be forward or backward. 
		\item The vertical arrows and arrows in zigzags forwarding left are all in $\WWW$. 
		\item And all the horizontal arrows in the same column (e.g., arrows between $C_{i, j}$ and $C_{i, j+1}$ for a fixed $i$) have the same direction. 
		\item Finally, each zigzag rows are restricted (i.e., no adjacent arrows with the same direction, no identities, leftward arrows are in $\WWW$).
	\end{enumerate}
	\item In good settings, i.e. where we have a homotopy calculus of right fractions, a simpler equivalent model exists:
	Define $\HHH^{-1}\CCC$ be the simplicial
	category with 
	\[\mathrm{Obj}(\HHH^{-1}\CCC) = \mathrm{Obj}(\CCC),\;\;\;\;\;\; \mathrm{Mor}_{\HHH^{-1}\CCC} (x, y) = \NNN \mathrm{Span}(\CCC, \HHH),\]
	where objects in $\mathrm{Span}(\CCC, \HHH)$ are of the format $X\xleftarrow{f\in \HHH} Z\xrightarrow{g} Y$, and morphisms are of the form $\phi$ in $\CCC$ such that the following diagram commutes:
	\begin{center}
		\begin{tikzcd}
			X & Z\arrow[l, "f\in \HHH"']\arrow[r, "g"]\arrow[d, "\phi"] & Y\\
			& Z'\arrow[ul, "f'\in \HHH"]\arrow[ur, "g'"']&
		\end{tikzcd}
	\end{center}
	$\NNN$ here means the nerve.
	(This is given by the localization of the type $w$ where $w$ is a word of alphabets in $\{\CCC, \WWW^{-1}\}$, also see Dwyer and Kan \cite{DK80b}.)
\end{enumerate}

Then the following is a restatement of Theorem 2.13 in \cite{RZ20} for our context, which related the above two models of simplicial models in our context:
\begin{theorem}\label{Equivalence of Two Localizations} Suppose $\CCC$ is a CFO with a subcategory $\WWW$ of weak equivalences and a subcategory $\HHH$ of trivial fibrations.
	The (canonical) functor $\HHH^{-1}\CCC\to \LLL(\CCC, \WWW)$ is an equivalence of simplicial categories.
\end{theorem}
%%%%%%%%%%%%%%%%%%
%%%%%%%%%%%%%%%%%%
%%%%%%%%%%%%%%%%%%
%%%%%%%%%%%%%%%%%%
%%%%%%%%%%%%%%%%%%
%%%%%%%%%%%%%%%%%%
%%%%%%%%%%%%%%%%%%

Then, one can define the ``$n$-stacks'':
\begin{definition}
	Suppose $\CCC$ is a category with covers and finite limits. The simplicial category of $n$-stacks in $\CCC$ is defined by $\HHH^{-1} (n\text{-}\mathrm{Grpds}(\CCC))$.
\end{definition}
Here we need to introduce a few of important examples for futher use. They are compatible with the literature, for example, Section 1.3.4 in Pridham \cite{P2013}. 
\begin{example}$ $ \label{Examples for n Stacks by Hammock Localizations}
	\begin{enumerate}
		\item $\CCC = \mathrm{Et}_X$, the catgeory of schemes \'etale surjective onto $X$, where $X$ is a connected scheme (over a field $k$). In this category, covers are surjective \'etale morphisms (over $X$). Then, by Lemma \ref{n groupoids form a CFO}, $n$-Grpds(FEt$_X$) is a CFO, where trivial fibrations $\HHH$ are the same as hypercovers, which is defined by Lemma \ref{Definition_Hypercovers}. Then, by the definition above, the category of $n$-stacks in $\mathrm{FEt}_X$ is given by 
		\[n\text{-}\mathrm{Stacks}(\mathrm{Et}_X) = \HHH^{-1}(n\text{-}\mathrm{Grpds}(\mathrm{Et}_X)).\]
		For usual stacks, using groupoids, we may define Deligne-Mumford stacks as (equivalence classes of) $1$-groupoid schemes in $\mathrm{Et}_X$, and the morphisms are given by localizing the category of $1$-groupoid schemes with respect $1$-hypercovers (c.f. Wolfson \cite{W16}, Theorem 2.17(1)). Note that, $n$-groupoids \'etale over $X$ implies the degeneracy and face maps are all \'etale, so we may define the (simplicial) category of DM-$n$-stacks \'etale over $X$ by 
		\[\mathrm{DM}\text{-}n\text{-}\mathrm{Stacks}/\mathrm{\acute{e}tale}\text{ } \mathrm{over}\text{ }X = n\text{-}\mathrm{Stacks}(\mathrm{Et}_X).\]
		See Section 1.3.4 in \cite{P2013} for more details.
		\item 
		Similarly, when $\CCC = \pi_1^{\mathrm{\acute{e}t}}(X, \eta)\text{-}\mathrm{FinSets}$, in this category, covers are surjective $\pi_1^{\mathrm{\acute{e}t}}(X, \eta)\text{-}\mathrm{FinSets}$-equivariant maps. Then the category of $1$-stacks in $\pi_1^{\mathrm{\acute{e}t}}(X, \eta)\text{-}\mathrm{FinSets}$ is defined by 
		$$1\text{-}\mathrm{Stacks}(\pi_1^{\mathrm{\acute{e}t}}(X, \eta)\text{-}\mathrm{FinSets}) = \HHH^{-1}(\mathrm{Grpds}(\pi_1^{\mathrm{\acute{e}t}}(X, \eta)\text{-}\mathrm{FinSets})).$$ 
		Here $\HHH$ are still the $1$-hypercovers. We may rewrite
		\[1\text{-}\mathrm{Stacks}(\pi_1^{\mathrm{\acute{e}t}}(X, \eta)\text{-}\mathrm{FinSets})  
		= 
		\pi_1^{\mathrm{\acute{e}t}}(X, \eta)\text{-} 1\text{-}\mathrm{Stacks}(\mathrm{FinSets}).\]
		
		\item In the schematic Galois correspondence like Theorem \ref{SchemeGaloisCorrespondence}, one side of the category equivalence is $\mathrm{FEt}_X$ rather than $\mathrm{Et}_X$, so we are also curious about the situation for $\CCC = \mathrm{FEt}_X$. By the schematic Galois correspondence, since the covers in the category $\pi_1^{\mathrm{\acute{e}t}}(X, \eta)\text{-}\mathrm{FinSets}$ are surjective $\pi_1^{\mathrm{\acute{e}t}}(X, \eta)\text{-}\mathrm{FinSets}$-equivariant maps, the corresponding morphisms in $\CCC = \mathrm{FEt}_X$ are (surjective) finite \'etale covers (see Theorem \ref{SchemeGaloisCorrespondence}). So in $\CCC = \mathrm{FEt}_X$, covers that are all (surjective) finite \'etale covers. The collection of hypercovers $\HHH$ is defined in Lemma \ref{Definition_Hypercovers}.
		Then the category of $n$-stacks in $\mathrm{FEt}_X$ is given by 
		\[n\text{-}\mathrm{Stacks}(\mathrm{FEt}_X)  = \HHH^{-1} (\mathrm{FEt}_X).\]
	\end{enumerate}
\end{example}
\begin{remark}[Remark of (iii) above]\label{Remark_Subtlety obstruct the category correspondence}
	Note that this is subtly different from a full subcategory of $\mathrm{DM}\text{-}n\text{-}\mathrm{Stacks}/\mathrm{\acute{e}tale}\text{ } \mathrm{over}\text{ }X$, since for \'etale surjective morphisms between two finite \'etale schemes $Y_1$ and $Y_2$ over a connected scheme $X$, it is only guaranteed to be quasi-finite instead of finite. However, the $1$-morphisms in $\Hom(X, Y)$ (in $n$-Stacks(Et$_X$)) may be given by $X\xleftarrow[\sim]{f} Z\xrightarrow{g} Y$ where $f$ can be just \'etale but not finite ( i.e., not proper). So we do not simply say $n\text{-}\mathrm{Stacks}(\mathrm{FEt}_X)$ is equal to the full subcategory of DM-$n$-stacks finite \'etale over $X$ in $n$-Stacks(Et$_X$). This is why our next result is not an equivalence of two $\infty$-categories.
\end{remark}
We also give a brief remark on different interpretations of usual DM-stacks below.
\begin{remark}[Equivalence of Different Interpretations of DM-Stacks]\label{remark equivalence of different interpretations of DM-stacks}
	One can actually verify that, if one takes $n = 1$, the above DM-$1$-stacks are compatible with the DM-stacks defined as fibered categories. Let's take $\CCC = \mathrm{Et}(S)$, the category of schemes over a (nice) scheme $S$; take the covers to be surjective \'etale covers; and take $n = 1$. Use the terminology in \cite{LMB18}, objects in DM-$1$-Stacks/\'etale over $S$ in this notes are \textit{groupoids in $S$-schemes}.
	\begin{itemize}
		\item  Suppose $\mathcal{X}$ is a DM-stack \'etale over $S$ in the sense of Definition 4.1 in \cite{LMB18}, by definition, one can obtain a presentation, a surjective \'etale morphism $p: X_0\to \mathcal{X}$. (Here $X_0$ can be taken as a scheme.) Then one can immediately obtain a groupoid in $S$-schemes
		\[(X_1:= X_0\times_{p, \mathcal{X}, p} X_0)\rightrightarrows X_0.\]
		%Note that $p$ is finite, so it's quasi-finite, thus, $(p, p): X_1\to X_0\times_S X_0$ is quasi-finite, so it is separated and quasi-compact.
		
		\item On the other hand, given a groupoid in schemes \'etale over $S$ \begin{tikzcd}
			X_1 \arrow[r, "d_0", shift left]
			\arrow[r, "d_1"', shift right]
			& X_0
		\end{tikzcd} with the 3 conditions in Proposition 4.3.1 in \cite{LMB18}, we can define an associated prestack $[X_1\rightrightarrows X_0]^{\mathrm{pre}}$, for example, we can let:
		\begin{itemize}
			\item Objects of $[X_1\rightrightarrows X_0]^{\mathrm{pre}}$ over $V\in \mathrm{Sch}/_S$ are given by $X_0(V)$;
			\item Given $g'\in X_0(V')$ and $g\in X_0(V)$, a morphism $f: g\to g'$ is a morphism $f: V'\to V$ such that $gf = g'$, and such that there exists a morphism $\gamma\in X_1(V')$ with $d_0\circ \gamma = g'$ and $d_1\circ \gamma = g$.
		\end{itemize}
		Then the stack associated to the groupoid $[X_1\rightrightarrows X_0]$ is given by the stackification of $[X_1\rightrightarrows X_0]^{\mathrm{pre}}$, see Lemma 3.2 in \cite{LMB18}. 
	\end{itemize}
	By Proposition 4.3.1 in \cite{LMB18}, $X_0\to [X_1\rightrightarrows X_0]$ gives a presentation. 
	By Proposition 3.8 in \cite{LMB18}, $\mathcal{X}$ with presentation $p: X_0\to \mathcal{X}$ has an isomorphism $\mathcal{X}\cong [X_1\rightrightarrows X_0]$. Thus, by the above, the category of DM-stacks \'etale over $S$ (using the definition of fibered categories) is equivalent to the category of DM-$1$-stacks in Et$(S)$ whose covers are surjective \'etale  covers.
\end{remark}

Then, Corollary \ref{corollary of main thm}, restated as below, can be considered as a immediate corollary if we apply the simplicial localizations on $F_\eta$ with respect to hypercovers: 
\begin{corollary}[Galois Correspondence for $n$-stacks]\label{Corollary_Main Corollary Proof}
	There is an induced essentially surjective functor of simplicial categories
	\[G^{\LLL}_\eta: 
	\pi^{\mathrm{\acute{e}t}}_1(X, \eta)\text{-}n\text{-}\mathrm{Stacks(FinSets)}\to\mathrm{DM}\text{-}n\text{-}\mathrm{Stacks}/\mathrm{\acute{e}tale}\text{ } \mathrm{over}\text{ }X  ,\]
	where $\pi^{\mathrm{\acute{e}t}}_1(X, \eta)\text{-}n\text{-}\mathrm{Stacks(FinSets)}$ is defined in (ii) of Example \ref{Examples for n Stacks by Hammock Localizations}, and $\mathrm{DM}\text{-}n\text{-}\mathrm{Stacks}/\mathrm{\acute{e}tale}\text{ } \mathrm{over}\text{ }X$ is the same as the localization $\mathrm{Stacks(}\mathrm{Et}_X\mathrm{)}$ that we defined in Example \ref{Examples for n Stacks by Hammock Localizations}.
\end{corollary}
% \begin{proof}
	%     Apply the hammock localizations on the functor $F_\eta$ in Theorem \ref{Galois Correspodence n Groupoids}, we get a functor $$F^{\LLL}_\eta: \mathrm{DM}\text{-}n\text{-}\mathrm{Stacks}/\mathrm{etale}\text{ } \mathrm{over}\text{ }X \to \pi^{\mathrm{et}}_1(X, \eta)\text{-}n\text{-}\mathrm{Stacks(FinSets)}.$$ 
	%     Notice that $F^{\mathrm{Sch}}_\eta$ in Theorem \ref{SchemeGaloisCorrespondence} is actually an equivalences on covers in FEt$_X$ and $\pi_1^{\mathrm{et}}(X, \eta)$-$\mathrm{FinSets}$, then there exists an equivalence between weak equivalence $\WWW\subset n\text{-}\mathrm{Grpds}(\mathrm{FEt}_X)$ and $\WWW'\subset \pi_1^{\mathrm{et}}(X, \eta)\text{-}\mathrm{Grpds}(\mathrm{FinSets})$. By Corollary \ref{Equivalence of Two Simplial Localizations from two CFOs}, we get the result.
	% \end{proof}

\begin{proof} The proof is roughly compositing a few of functors. 
	
	Since $F_\eta$ in the Galois Correspondence restricts to an equivalence of weak equivalences, we may denote $G_\eta$ as the inverse.
	
	By Theorem \ref{Equivalence of Two Localizations} and Theorem \ref{Galois Correspodence n Groupoids}, we get 
	\begin{align*}
		\LLL(\mathrm{Grpds}(\pi_1^{\mathrm{\acute{e}t}}(X, \eta)\text{-}\mathrm{FinSets}), \WWW_S) \xrightarrow{\simeq} \LLL(\mathrm{Grpds}(\mathrm{FEt}_X), \WWW_{\mathrm{FEt}})
	\end{align*}
	where $\WWW_S$ is the collection of weak equivalences in the $n$-groupoid objects in $\pi_1^{\mathrm{\acute{e}t}(X, \eta)}$-FinSets, and $\WWW_{\mathrm{FEt}}$ is the collection of weak equivalences in the $n$-groupoid objects in $\mathrm{FEt}_X$. See (ii) and (iii) in Example \ref{Examples for n Stacks by Hammock Localizations}. 
	
	Note that we can use the simpler model of simplicial localization with hypercovers, then we get 
	\begin{align}\GGG_1: \HHH_{\mathrm{Surj}}^{-1}\mathrm{Grpds}(\pi_1^{\mathrm{\acute{e}t}}(X, \eta)\text{-}\mathrm{FinSets}) \xrightarrow{\simeq}
		\HHH_{\mathrm{FEt}}^{-1}\mathrm{Grpds}(\mathrm{FEt}_X),
	\end{align}
	where $\HHH_{\mathrm{Surj}}$ is the hypercovers given by covers being surjective $\pi_1^{\mathrm{\acute{e}t}}(X, \eta)$-equivariant surjective maps, see (ii) in Example \ref{Examples for n Stacks by Hammock Localizations}; and where $\HHH_{\mathrm{FEt}}$ is the hypercovers given by covers being (surjective) \'etale finite morphisms, see (iii) in Example \ref{Examples for n Stacks by Hammock Localizations}.
	
	We have a canonical functor $\HHH_{\mathrm{FEt}}^{-1}\mathrm{Grpds}(\mathrm{FEt}_X)\to \HHH_{\mathrm{Et}}^{-1}\mathrm{Grpds}(\mathrm{Et}_X)$, by sending $n$-groupoid objects in $\mathrm{FEt}_X$ to itself in $\mathrm{Grpds}(\mathrm{Et}_X)$. Besides, ``$\mathrm{DM}\text{-}n\text{-}\mathrm{Stacks}/\mathrm{finite}\text{ }\mathrm{\acute{e}tale}\text{ } \mathrm{over}\text{ }X$'' can be considered as a subcategory of ``DM-$n$-stacks \'etale over $X$'' whose objects are all $n$-stacks in $\mathrm{FEt}_X$ with less hypercovers. So we may further refine the functor above and get an essentially surjective functor
	to ``$\mathrm{DM}\text{-}n\text{-}\mathrm{Stacks}/\mathrm{finite}\text{ }\mathrm{\acute{e}tale}\text{ } \mathrm{over}\text{ }X $'' as below: \begin{align}\HHH_{\mathrm{FEt}}^{-1}\mathrm{Grpds}(\mathrm{FEt}_X)&\xrightarrow{\GGG_2, \:\simeq}\mathrm{DM}\text{-}n\text{-}\mathrm{Stacks}/\mathrm{finite}\text{ }\mathrm{\acute{e}tale}\text{ } \mathrm{over}\text{ }X\\
		&\nonumber\hookrightarrow \mathrm{DM}\text{-}n\text{-}\mathrm{Stacks}(\mathrm{Et}_X) \xrightarrow{\simeq} \HHH_{\mathrm{Et}}^{-1}\mathrm{Grpds}(\mathrm{Et}_X).
	\end{align}
	By Example \ref{Examples for n Stacks by Hammock Localizations}, we are done.
	%Define $\GGG_\eta^\LLL:= \GGG_2\circ \GGG_1$, we proved what we want.
	
\end{proof}

\section{Topological and Smooth Analogues}

The classical Galois correspondence has a topological version as below, see Corollary 2.3.9 in \cite{S09}: 
\begin{theorem}\label{Topological Classical Galois Correspondence}
	Let $X$ be a connected, locally connected and semi-locally simply connected topological space, and let $x\in X$ be a point. Let \normalfont{FCov}$(X)$ \textit{be the category of finite (usual) covers of $X$, and, let} $\pi_1(X, x)$-\normalfont{FinSets}$ $ \textit{be the category of finite $\pi_1(X, x)$-sets. 
		Then the following functor, called the fiber functor,}
	\[F_x: \mathrm{FCov}(X)\to \pi_1(X, x)\text{-}\mathrm{FinSets},\;\; (Y\xrightarrow{f} X) \mapsto f^{-1}(x), \]
	\textit{is an equivalence of categories.}
\end{theorem}

Now, the category FCov$(X)$ of finite (usual topological) covers of $X$ can be also considered as a category with covers and finite limits. In particular, covers in this category are just the surjective local homeomorphisms. 
%( One may also consider the category of differentiable covers of a connected smooth manifold $X$, then the "covers" in the category will be surjective local diffeomorphisms. )

The arguments above carry over to the present setting to give the following parallel results.

\begin{theorem}[Galois Correspondence for $n$-groupoids in Topological Spaces]\label{Topological main theorem}
	Let $X$ and $x$ as in Theorem \ref{Topological Classical Galois Correspondence}. There is an exact equivalence of categories of fibrant objects
	\[F_x: n\text{-}\mathrm{Grpds}(\mathrm{FCov }(X))\xrightarrow{\sim} \pi_1(X, x)\text{-}n\text{-}\mathrm{Grpds}(\mathrm{FinSets}).\]
	Here $n\text{-}\mathrm{Grpds}(\mathrm{FCov}(X))$ is the $1$-category of $n$-groupoids in the category $\mathrm{FCov}_X$, the category of finite topological covers over $X$; and $\pi_1(X, x)\text{-}n\text{-}\mathrm{Grpds}(\mathrm{FinSets})$ is the category of $n$-groupoids in finite sets with $\pi_1(X, \eta)$ acting on them. 
\end{theorem}

Then, use the same $\HHH_{\mathrm{Surj}}$ in the proof of Corollary \ref{Corollary_Main Corollary Proof}, replace $\HHH_{\mathrm{Et}}$ by $\HHH_{\mathrm{Cov}}$ which is all the hypercovers in ``$n$-Grpds(Cov$(X)$)'' where the covers in Cov$(X)$ are all surjective topological finite covers, and also replace $\HHH_{\mathrm{FEt}}$ by $\HHH_{\mathrm{FCov}}$ which is the collection of all hypercovers defined by covers being surjective finite topological covers among covers of $X$, we obtained the following.
% \begin{corollary}\label{Topological corollary of main thm}
	%     Let $X$ and $x$ as in Theorem \ref{Topological Classical Galois Correspondence}. There exists an  equivalence of simplicial categories
	%     \[F^\LLL_x: \mathrm{DM}\text{-}n\text{-}\mathrm{Stacks}/\mathrm{Loc.}\text{ }\mathrm{Homeo.} \text{ }\mathrm{over}\text{ }X \xrightarrow{\sim} \pi_1(X, x)\text{-}n\text{-}\mathrm{Stacks(FinSets)}.\]
	%     Here, $\mathrm{DM}\text{-}n\text{-}\mathrm{Stacks}/\mathrm{Loc.}\text{ }\mathrm{Homeo.} \text{ }\mathrm{over}\text{ }X$ is given by hammock localization of \'etale $n$-groupoids in $\mathrm{FCov}(X)$ at its hyperpcovers; and $\pi_1(X, \eta)\text{-}n\text{-}\mathrm{Stacks(FinSets)}$ is obtained by hammock localization of $\pi_1(X, \eta)\text{-}n\text{-}\mathrm{Groupoids}$.
	% \end{corollary} 

\begin{corollary}\label{Topological corollary of main thm}
	Let $X$ and $x$ as in Theorem \ref{Topological Classical Galois Correspondence}. There exists an essentially surjective functor
	\[G^\LLL_x:  \pi_1(X, x)\text{-}n\text{-}\mathrm{Stacks(FinSets)} \rightarrow\mathrm{DM}\text{-}n\text{-}\mathrm{Stacks}/\mathrm{Loc.}\text{ }\mathrm{Homeo.} \text{ }\mathrm{over}\text{ }X
	.\]
	Here, $\mathrm{DM}\text{-}n\text{-}\mathrm{Stacks}/\mathrm{Loc.}\text{ }\mathrm{Homeo.} \text{ }\mathrm{over}\text{ }X$ is given by hammock localization of \'etale $n$-groupoids in $\mathrm{FCov}(X)$ at its hyperpcovers obtained by considering covers in the category of covers of $X$ to be surjective finite topological covers; and $\pi_1(X, \eta)\text{-}n\text{-}\mathrm{Stacks(FinSets)}$ is obtained by hammock localization of $\pi_1(X, \eta)\text{-}n\text{-}\mathrm{Groupoids}$.
\end{corollary} 

We can also tell a parallel story in the realm of (connected) smooth manifolds and its finite smooth covers. Specifically, we take $\mathrm{FDiff}(X)$ to be the category of finite smooth covers of a smooth manifold $X$, and covers in $\mathrm{FDiff}(X)$ to be the surjective local diffeomorphisms. Then we also have the following Galois correspondence for finite smooth covers.
\begin{theorem}\label{Smooth Classical Galois Correspondence}
	Suppose $X$ is a connected smooth manifold, and $x\in X$ is a fixed point. Then the following functor, called the fiber functor,
	\[F_{x}: \mathrm{FDiff}(X)\to \pi_1(X, x)\text{-}\mathrm{FinSets}, \:\:\:\: (Y\xrightarrow{f} X)\mapsto f^{-1}(x),\]
	is an equivalence of categories.
\end{theorem}

As before, we also have the following analogous result as above.
\begin{theorem}[Galois Correspondence for $n$-Groupoids in Smooth Manifolds]
	\label{smooth main theorem}
	Let $X$ and $x$ as in Theorem \ref{Smooth Classical Galois Correspondence}. There is an exact equivalence of categories of fibrant objects
	\[F_x: n\text{-}\mathrm{Grpds}(\mathrm{FDiff }(X))\xrightarrow{\sim} \pi_1(X, x)\text{-}n\text{-}\mathrm{Grpds}(\mathrm{FinSets}).\]
	Here $n\text{-}\mathrm{Grpds}(\mathrm{FDiff}(X))$ is the $1$-category of $n$-groupoids in the category $\mathrm{FDiff}_X$, the category of finite smooth covers over $X$; and $\pi_1(X, x)\text{-}n\text{-}\mathrm{Grpds}(\mathrm{FinSets})$ is the category of $n$-groupoids in finite sets with $\pi_1(X, \eta)$ acting on them. 
\end{theorem}

Then, as we did in Example \ref{Examples for n Stacks by Hammock Localizations} for $\mathrm{Et}_X$, by Lemma \ref{n groupoids form a CFO}, $n$-Grpds($\mathrm{FDiff}(X)$) is a CFO, where trivial fibrations $\HHH$ are the same as hypercovers. Then,
\[n\text{-}\mathrm{Stacks}(\mathrm{FDiff}(X)) = \HHH^{-1}(n\text{-Grpds}(\mathrm{FDiff}(X))).\]
Similar to the definition of $n$-Stacks($\mathrm{Et}_X$) in Example \ref{Examples for n Stacks by Hammock Localizations}, we may define 
\[\mathrm{DM}\text{-}n\text{-}\mathrm{Stacks}/\mathrm{Loc.}\text{ }\mathrm{Diffeo.} \text{ }\mathrm{over}\text{ }X= n\text{-}\mathrm{Stacks}(\mathrm{FDiff}(X)).\]
Suppose we define the category of smooth covers of $X$ to be $\mathrm{Diff}(X)$.
Then, use the same $\HHH_{\mathrm{Surj}}$ in the proof of Corollary \ref{Corollary_Main Corollary Proof}, replace $\HHH_{\mathrm{Et}}$ by $\HHH_{\mathrm{Diff}}$ which is all the hypercovers in ``$n$-Grpds(Diff$(X)$)'' where the covers in Diff$(X)$ are all surjective smooth covers over $X$, and also replace $\HHH_{\mathrm{FEt}}$ by $\HHH_{\mathrm{FDiff}}$ which is the collection of all hypercovers defined by covers being surjective smooth finite covers among covers of $X$, we obtained the following, we can get the following theorem as an analogue of Corollary \ref{corollary of main thm}.

\begin{corollary}\label{Topological corollary of main thm diffeo}
	Let $X$ be a smooth manifold, $x\in X$ is a fixed point. There exists an essentially surjective functor of simplicial categories
	\[G^\LLL_x:  \pi_1(X, x)\text{-}n\text{-}\mathrm{Stacks(FinSets)}\to\mathrm{DM}\text{-}n\text{-}\mathrm{Stacks}/\mathrm{Loc.}\text{ }\mathrm{Diffeo.} \text{ }\mathrm{over}\text{ }X
	.\]
	Here, $\mathrm{DM}\text{-}n\text{-}\mathrm{Stacks}/\mathrm{Loc.}\text{ }\mathrm{Diffeo.} \text{ }\mathrm{over}\text{ }X$ is given by hammock localization of \'etale $n$-groupoids in $\mathrm{FDiff}(X)$ at its hyperpcovers obtained by considering covers in the category of smooth covers of $X$ to be surjective smooth covers; and $\pi_1(X, \eta)\text{-}n\text{-}\mathrm{Stacks(FinSets)}$ is obtained by applying hammock localization of $\pi_1(X, \eta)\text{-}n\text{-}\mathrm{Groupoids}$.
\end{corollary}

%\DeclareFieldFormat{labelalphawidth}{#1}
%\DeclareFieldFormat{shorthandwidth}{#1}
%
\nocite{*}
\printbibliography

\end{document}